\documentclass[a4paper,11pt]{article}
\usepackage{graphicx}
\usepackage[utf8x]{inputenc}
\usepackage{fullpage}
\usepackage{amsmath}
\usepackage{amssymb}
\usepackage{theorem}
\usepackage{stmaryrd}
\usepackage{hyperref}

\newtheorem{theorem}{Theorem}
\newtheorem{lemma}[theorem]{Lemma}
\newtheorem{proposition}[theorem]{Proposition}
\newtheorem{corollary}[theorem]{Corollary}

\theorembodyfont{\rm} 
\newtheorem{definition}[theorem]{Definition}

\newenvironment{proof}{\noindent {\it Proof}.}{\hfill$\Box$}

\newcommand{\lignelarge}{\lower1.5ex\hbox{\rule{0ex}{4ex}}}
\let\disp\displaystyle
\newcommand{\IZ}{{\mathbb Z}}

\let\Lbrack\llbracket
\let\Rbrack\rrbracket

\begin{document}

\title{Topological Markov chains of given entropy and period with or without measure of maximal entropy}
\author{Sylvie Ruette}
\date{June 1, 2018}

\maketitle

\begin{abstract}
We show that, for every positive real number $h$ and every positive
integer $p$, there exist oriented graphs $G, G'$ (with countably 
many vertices) that are strongly connected, of period $p$, of
Gurevich entropy $h$, such
that $G$ is positive recurrent
(thus the topological Markov chain on $G$ admits a measure of maximal 
entropy)  and $G'$ is transient
(thus the topological Markov chain on $G'$
admits no measure of maximal entropy).
\end{abstract}

\section{Vere-Jones classification of graphs}

In this paper, all the graphs are oriented, have a finite or 
countable set of vertices and, if $u,v$ are two vertices, there
is at most one arrow $u\to v$. 
A {\em path} of length $n$ in the graph $G$ is a sequence of vertices 
$(u_0,u_1,\ldots,u_n)$ such that $u_i\to u_{i+1}$
in $G$ for all $i\in\Lbrack 0,n-1\Rbrack$.
This path is called a {\em loop} if $u_0=u_n$.

\begin{definition}
Let $G$ be an oriented graph and let $u,v$ be two vertices in $G$. 
We define the following quantities.
\begin{itemize}
\item $p_{uv}^G(n)$ is the number of paths $(u_0,u_1,\ldots,u_n)$ 
such that $u_0=u$ and $u_n=v$; $R_{uv}(G)$ is the radius of convergence
of the series $\sum p_{uv}^G(n)z^n$.
\item $f_{uv}^G(n)$ is the number of paths $(u_0,u_1,\ldots,u_n)$ 
such that $u_0=u$, $u_n=v$ and $u_i\not=v$ for all $0<i<n$; 
$L_{uv}(G)$ is the radius of convergence
of the series $\sum f_{uv}^G(n)z^n$.
\end{itemize}
\end{definition}

\begin{definition}
Let $G$ be an oriented graph and  $V$ its set of vertices.
The graph $G$ is \emph{strongly connected} if for all $u,v\in V$, 
there exists a path from $u$ to $v$ in $G$.
The \emph{period} of a strongly connected graph $G$ is the
greatest common divisor of $(p_{uu}^G(n))_{u\in V, n\ge 0}$. The graph
$G$ is \emph{aperiodic} if its period is $1$.
\end{definition}

\begin{proposition}[Vere-Jones \cite{Ver1}]
Let $G$ be an  oriented graph. If $G$ is strongly connected, 
$R_{uv}(G)$ does not depend on $u$ and $v$; it is denoted by $R(G)$.
\end{proposition}

If there is no confusion, $R(G)$ and $L_{uv}(G)$ will be written $R$ and 
$L_{uv}$. 

\medskip
In \cite{Ver1} Vere-Jones gives a classification of strongly connected graphs
as transient, null recurrent or positive recurrent. These definitions
are lines 1 and 2 in Table~\ref{tab:classification}. The other lines
of Table~\ref{tab:classification} state
properties of the series $\sum p_{uv}^G(n)z^n$, which give alternative
definitions (lines 3 and 4 are in \cite{Ver1}, 
the last line is Proposition~\ref{prop:transient-R=L}). 

\begin{table}[ht] 
\begin{center}
\begin{tabular}{l|c|c|c}
                     & transient   & null      & positive  \\
                     &             & recurrent & recurrent \\
\hline
\lignelarge
$\disp\sum_{n>0} f^G_{uu}(n)R^n$  & $<1$        & $1$       & $1$     \\
\hline
\lignelarge
$\disp\sum_{n>0} nf^G_{uu}(n)R^n$ &$\leq+\infty$&$+\infty$  &$<+\infty$ \\
\hline
\lignelarge
$\disp\sum_{n\geq 0} p^G_{uv}(n)R^n$  &$<+\infty$   &$+\infty$  &$+\infty$ \\
\hline
\lignelarge
$\disp\lim_{n\to+\infty} p^G_{uv}(n)R^n$  & $0$         & $0$       &$\lambda_{uv}>0$\\
\hline
\lignelarge
                     & $R=L_{uu}$       & $R=L_{uu}$     & $R\leq L_{uu}$
\end{tabular}
\end{center}
\caption{properties of the series associated to a transient, null 
recurrent or positive recurrent graph $G$ ($G$ is strongly connected); these
properties do not depend on the vertices $u,v$. \label{tab:classification}}
\end{table}

\begin{proposition}[Salama  \cite{Sal3}] \label{prop:transient-R=L}
Let $G$ be a strongly connected oriented graph. If $G$ is transient or null
recurrent, then $R=L_{uu}$ for all vertices $u$. Equivalently,
if there exists a vertex $u$ such that $R<L_{uu}$, then $G$ is positive 
recurrent.
\end{proposition}

\section{Topological Markov chains and Gurevich entropy}

Let $G$ be an oriented graph and $V$ its set of vertices. We define
$\Gamma_G$ as the set of two-sided infinite paths in $G$, that is,
\[
\Gamma_G:=\{(v_n)_{n\in\IZ} \mid  \forall n\in \IZ, v_n\to v_{n+1} 
\mbox{ in } G \}\subset V^{\IZ}.
\]
The map $\sigma$ is the shift on $\Gamma_G$. The 
\emph{topological Markov chain}
on the graph $G$ is the dynamical system $(\Gamma_G,\sigma)$.

The set $V$ is endowed with the discrete topology and $\Gamma_G$
is endowed with the induced topology of $V^{\IZ}$. 
The space $\Gamma_G$ is not compact unless $G$ is finite.

The topological Markov chain $(\Gamma_G,\sigma)$ is 
transitive if and only if the graph $G$ is strongly connected.
It is topologically mixing if and only if the graph $G$ is 
strongly connected and aperiodic.

\medskip
If $G$ is a finite graph, $\Gamma_G$ is compact and
the topological entropy $h_{top}(\Gamma_G,\sigma)$ is well defined
(see e.g. \cite{DGS} for the definition of the topological entropy).
If $G$ is a countable graph, the {\em Gurevich entropy} \cite{Gur1} 
of the graph $G$ (or of the topological Markov chain $\Gamma_G$)
is given by
\[
h(G):=\sup\{h_{top}(\Gamma_H,\sigma)\mid H\subset G, H \mbox{ finite}\}.
\]

This entropy can also be computed in a combinatorial way, as the exponential
growth of the number of paths with fixed endpoints.

\begin{proposition}[Gurevich \cite{Gur2}]\label{prop:hR}
Let $G$ be a strongly connected oriented graph. Then for all vertices $u,v$
\[
h(G)=\lim_{n\to+\infty}\frac{1}{n}\log p_{uv}^G(n)=-\log R(G).
\]
\end{proposition}

Moreover, the variational principle is still valid for topological
Markov chains.

\begin{theorem}[Gurevich \cite{Gur1}]
Let $G$ be an oriented graph. Then
\[
h(G)=\sup\{h_{\mu}(\Gamma_G)\mid \mu
\ \sigma\mbox{-invariant probability measure}\}.
\]
\end{theorem}

In this variational principle, the supremum is not necessarily reached.
The next theorem gives a necessary and sufficient condition for the existence 
of a measure of maximal entropy (that is, a probability measure $\mu$ such that
$h(G)=h_{\mu}(\Gamma_G)$) when the graph is strongly connected.

\begin{theorem}[Gurevich \cite{Gur2}]\label{theo:maxmeasure}
Let $G$ be a strongly connected oriented graph of finite positive entropy. 
Then the topological Markov chain on $G$ admits a measure of maximal
entropy if and only if the graph $G$ is positive recurrent. 
Moreover, such a measure is unique if it exists.
\end{theorem}

\section{Construction of graphs of given entropy and given period
that are either positive recurrent or transient}

\begin{lemma}\label{lem:an}
Let $\beta\in (1,+\infty)$. There exist a sequence of non negative integers
$(a(n))_{n\ge 1}$ and positive constants $c,M$ such that
\begin{itemize}
\item $a(1)= 1$,
\item $\sum_{n\ge 1}a(n)\frac{1}{\beta^n}=1$,
\item $\forall n\ge 2$, $c\cdot \beta^{n^2-n}\le a(n^2)\le c\cdot \beta^{n^2-n}
+M$,
\item $\forall n\ge 1$, $0\le a(n)\le M$ if $n$ is not a square.
\end{itemize}
These properties imply that the radius of convergence of
$\sum_{n\ge 1}a(n)z^n$ is $L=\frac{1}{\beta}$ and that
$\sum_{n\ge 1} n a(n)L^n<+\infty$.
\end{lemma}

\begin{proof}
First we look for a constant $c>0$ such that
\begin{equation}\label{eq:c}
\frac{1}{\beta}+c\sum_{n\ge 2}\beta^{n^2-n}\frac{1}{\beta^{n^2}}=1.
\end{equation}
We have
\[
\sum_{n\ge 2}\beta^{n^2-n}\frac{1}{\beta^{n^2}}=\sum_{n\ge 2}\beta^{-n}
=\frac{1}{\beta(\beta-1)}.
\]
Thus
\[
\eqref{eq:c}
\Longleftrightarrow \frac{1}{\beta}+\frac{c}{\beta(\beta-1)}=1\Longleftrightarrow 
c=(\beta-1)^2.
\]
Since $\beta>1$, the constant $c:=(\beta-1)^2$ is positive.
We define the sequence $(b(n))_{n\ge 1}$ by:
\begin{itemize}
\item $b(1):=1$,
\item $b(n^2):=\lfloor c \beta^{n^2-n}\rfloor$ for all $n\ge 2$,
\item $b(n):=0$ for all $n\ge 2$ such that $n$ is not a square.
\end{itemize}
Then 
\[
\sum_{n\ge 1}b(n)\frac{1}{\beta^n}\le \frac{1}{\beta}+
c\sum_{n\ge 2}\beta^{n^2-n}\frac{1}{\beta^{n^2}}=1.
\]
We set $\delta:=1-\sum_{n\ge 1}b(n)\frac{1}{\beta^n}\in [0,1)$
and $k:=\lfloor \beta^2\delta\rfloor$. Then
$k\le \beta^2\delta<k+1<k+\beta$, which implies that
$0\le\delta-\frac{k}{\beta^2}<\frac{1}{\beta}$.
We write the $\beta$-expansion of $\delta-\frac{k}{\beta^2}$
(see e.g. \cite[p 51]{DK} for the definition): 
there exist integers 
$d(n)\in\{0,\ldots, \lfloor \beta\rfloor\}$ such that
$\delta-\frac{k}{\beta^2}=\sum_{n\ge 1}d(n)\frac{1}{\beta^n}$.
Moreover, $d(1)=0$ because $\delta-\frac{k}{\beta^2}<\frac{1}{\beta}$.
Thus we can write
\[
\delta=\sum_{n\ge 2}d'(n)\frac{1}{\beta^n}
\]
where $d'(2):=d(2)+k$ and $d'(n):=d(n)$ for all $n\ge 3$.

We set $a(1):=b(1)$ and $a(n):=b(n)+d'(n)$ for all $n\ge 2$.
Let $M:=\beta+k$. We then have:
\begin{itemize}
\item $a(1)=1$,
\item $\sum_{n\ge 1}a(n)\frac{1}{\beta^n}=1$,
\item $\forall n\ge 2$, $c\cdot \beta^{n^2-n}\le a(n^2)\le 
c\cdot \beta^{n^2-n}+\beta\le c\cdot \beta^{n^2-n}+M$, 
\item $0\le a(2)\le\beta+k=M$,
\item $\forall n\ge 3$, $0\le a(n)\le \beta\le M$ if $n$ is not a square.
\end{itemize}
The radius of convergence $L$ of $\sum_{n\ge 1}a(n)z^n$ satisfies
\[
-\log L= \limsup_{n\to+\infty} \frac1n \log a(n)=\lim_{n\to+\infty} \frac{1}{n^2} \log a(n^2)=\log \beta
\quad\mbox{because }a(n^2)\sim c\beta^{n^2-n}.
\]
Thus $L=\frac{1}{\beta}$. Moreover,
\[
\sum_{n\ge 1}n a(n)\frac{1}{\beta^n}\le
M\sum_{n\ge 1}n\frac{1}{\beta^n}+c\sum_{n\ge 1} n^2 \beta^{n^2-n}
\frac{1}{\beta^{n^2}}=
M\sum_{n\ge 1}\frac{n}{\beta^n}+c\sum_{n\ge 1}\frac{n^2}{\beta^{n}}<+\infty.
\]
\end{proof}

\begin{lemma}[\cite{R2}, Lemma 2.4]\label{lem:recurrent-R}
Let $G$ be a strongly connected oriented graph and $u$ a vertex.
\begin{enumerate}
\item $R<L_{uu}$ if and only if $\,\sum_{n\geq 1}f_{uu}^G(n)L_{uu}^n>1$.
\item If $G$ is recurrent, then $R$ is the unique positive number $x$ such that
$\sum_{n\geq 1}f_{uu}^G(n)x^n=1$.
\end{enumerate}
\end{lemma}

\begin{proof} For (i) and (ii), use Table~\ref{tab:classification}
and the fact that 
$F(x)=\sum_{n\geq 1}f_{uu}^G(n)x^n$ is increasing for $x\in[0,+\infty[$.
\end{proof}

\begin{proposition}\label{prop:shift-entropiefixee}
Let $\beta\in (1,+\infty)$. There exist aperiodic strongly connected graphs 
$G'(\beta)\subset G(\beta)$ such that $h(G(\beta))=h(G'(\beta))=\log\beta$,
$G(\beta)$ is positive recurrent and $G'(\beta)$ is transient.
\end{proposition}

Remark: Salama proved the part of this proposition concerning
positive recurrent graphs in \cite[Theorem~3.9]{Sal2}.

\medskip
\begin{proof}
This is a variant of the proof of \cite[Example~2.9]{R2}.

Let $u$ be a vertex and let $(a(n))_{n\ge 1}$ 
be the sequence given by Lemma~\ref{lem:an} for $\beta$.
The graph $G(\beta)$ is composed of $a(n)$ loops of length $n$ based at the vertex
$u$ for all $n\geq 1$ (see Figure~\ref{fig:G-G'}). More precisely, 
define the set of vertices of $G(\beta)$ as 
\[
V:=\{u\}\cup\bigcup_{n=1}^{+\infty}\{v_k^{n,i} \mid i\in\Lbrack 1,a(n)\Rbrack, 
k\in\Lbrack 1, n-1\Rbrack\},
\]
where the vertices $v_k^{n,i}$ above are distinct.
Let $v_0^{n,i}=v_n^{n,i}=u$ for all $i\in\Lbrack 1,a(n)\Rbrack$. 
There is an arrow $v_k^{n,i}\to v_{k+1}^{n,i}$ for all 
$k\in\Lbrack 0, n-1\Rbrack, i\in\Lbrack 1, a(n)\Rbrack, 
n\geq 2$; there is an arrow $u\to u$;
and there is no other arrow in $G(\beta)$. The graph $G(\beta)$ is strongly connected  
and $f_{uu}^{G(\beta)}(n)=a(n)$ for all $n\geq 1$. 

\begin{figure}[ht]
\centerline{\includegraphics{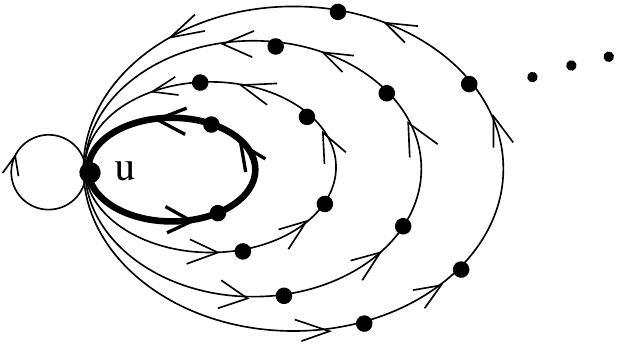}}
\caption{the graphs $G(\beta)$ and $G'(\beta)$; the bold loop belongs to $G(\beta)$ and not to $G'(\beta)$, 
otherwise the
two graphs coincide. \label{fig:G-G'}}
\end{figure}

By Lemma~\ref{lem:an}, the sequence $(a(n))_{n\geq 1}$ is defined such that
$L=\frac{1}{\beta}$ and
\begin{equation}\label{eq:=1}
\sum_{n\geq 1} a(n)L^n=1,
\end{equation}
where $L=L_{uu}(G(\beta))$ is the radius of convergence of the series 
$\sum a(n)z^n$.
If $G(\beta)$ is transient, then $R(G(\beta))=L_{uu}(G(\beta))$ by 
Proposition~\ref{prop:transient-R=L}. But 
Equation~\eqref{eq:=1} contradicts the definition of transient
(see the first line of Table~\ref{tab:classification}). Thus
$G(\beta)$ is recurrent, and $R(G(\beta))=L$ by Equation~\eqref{eq:=1}
and Lemma~\ref{lem:recurrent-R}(ii). Moreover
\[
\sum_{n\geq 1}na(n)L^n<+\infty
\]
by Lemma~\ref{lem:an}, and thus the graph $G(\beta)$ is positive recurrent
(see Table~\ref{tab:classification}). 
By Proposition~\ref{prop:hR},
$h(G(\beta))=-\log R(G(\beta))=\log \beta$.

The graph $G'(\beta)$ is obtained from $G(\beta)$ by deleting a
loop starting at $u$ of length $n_0$ for some $n_0\ge 2$ such that
$a(n_0)\ge 1$ (such an integer $n_0$ exists because $L<+\infty$).
Obviously one has $L_{uu}(G'(\beta))=L$ and
\[
\sum_{n\geq 1} f_{uu}^{G'(\beta)}(n)L^n=1-L^{n_0}<1.
\]
Since $R(G'(\beta))\leq L_{uu}(G'(\beta))$, this implies that 
$G'(\beta)$ is transient. Moreover
$R(G'(\beta))=L_{uu}(G'(\beta))$ by Proposition~\ref{prop:transient-R=L},
so $R(G'(\beta))=R(G(\beta))$, and hence $h(G'(\beta))=h(G(\beta))$
by Proposition~\ref{prop:hR}.
Finally, both $G(\beta)$ and $G'(\beta)$ are of period $1$ because of
the arrow $u\to u$.
\end{proof}

\begin{corollary}\label{cor:shift-entropiefixee}
Let $p$ be a positive integer and $h\in (0,+\infty)$. 
There exist strongly connected graphs $G,G'$ of period $p$ such that
$h(G)=h(G')=h$, $G$ is positive recurrent and $G'$ is transient.
\end{corollary}

\begin{proof}
For $G$, we start from the graph $G(\beta)$ given by
Proposition~\ref{prop:shift-entropiefixee} with $\beta=e^{hp}$. 
Let $V$ denote the set of vertices of $G(\beta)$.
The set of vertices of $G$ is $V\times\Lbrack 1,p\Rbrack$, and the
arrows in $G$ are:
\begin{itemize}
\item $(v,i)\to (v,i+1)$ if $v\in V$, $i\in\Lbrack 1,p-1\Rbrack$,
\item$(v,p)\to (w,1)$ if $v,w\in V$ and $v\to w$ is an arrow
in $G(\beta)$.
\end{itemize}
According to the properties of $G(\beta)$, $G$ is 
strongly connected, of period $p$ and positive recurrent. 
Moreover, $h(G)=\frac1p h(G(\beta))=\frac1p\log \beta=h$.

For $G'$, we do the same starting with $G'(\beta)$.
\end{proof}

\medskip
According to Theorem~\ref{theo:maxmeasure}, the graphs of
Corollary~\ref{cor:shift-entropiefixee} satisfy that
the topological Markov chain on $G$ admits a measure of maximal 
entropy whereas the topological Markov chain on $G'$
admits no measure of maximal entropy; both are transitive,
of Gurevich entropy $h$ and supported by a graph of period $p$.

\bibliographystyle{plain}

\end{document}